\documentclass[a4paper,11pt,reqno]{amsart}
\usepackage[american]{babel}
\usepackage{amssymb}
\usepackage{amsmath}
\usepackage{amsthm}
\usepackage{cleveref}
\usepackage{latexsym}
\usepackage{mathabx}
\usepackage{enumerate}
\usepackage{enumitem}
\usepackage{tikz-cd}
\usepackage[a4paper, margin=1in]{geometry}
\usepackage{color,soul}

\DeclareMathOperator{\Cay}{Cay}
\swapnumbers 
\theoremstyle{plain}
\newtheorem{thm}[subsection]{Theorem}
\theoremstyle{definition} 
\newtheorem{lem}[subsection]{Lemma}
\newtheorem{cor}[subsection]{Corollary}

\newtheorem{exa}[subsection]{Example}

\newtheorem{rem}[subsection]{Remark}
\newtheorem{theorem*}{Main Result}
\newtheorem{cor*}[theorem*]{Corollary}
\numberwithin{equation}{section}

\begin{document}
\title{On the Isomorphism Problem of Cayley Graphs of Graph Products}
\author{Marjory Mwanza}
\address{Marjory Mwanza\\ Department of Mathematics, M\"unster University, Germany}
\email{mmwanza@uni-muenster.de}
\thanks{The author is funded by the Deutsche Forschungsgemeinschaft (DFG, German Research Foundation) under Germany's Excellence Strategy EXC 2044 --390685587, Mathematics Münster: Dynamics--Geometry--Structure.}

\maketitle

\begin{abstract}
We investigate Cayley graphs of graph products by showing that graph products with vertex groups that have isomorphic Cayley graphs yield isomorphic Cayley graphs. 
\end{abstract}

\section{Introduction}
Graph products of groups are generalizations of both free and direct products. They are used to define a product of groups in a way that some factors may commute, but others not. The name graph product stems from the fact that the data about these products is encoded in a graph. The groups over which we are taking the product are called the vertex groups. In the graph, each vertex will represent a vertex group, and every edge will represent a relation. If the graph has no edges, the graph product coincides with the free product, and if the graph is complete, the graph product coincides with the direct product. More precisely, let $\Gamma = (V, E)$ be a graph. The graph product of groups $G_{\Gamma}$ is the group obtained from the free product of the vertex groups $G_v$ for $v \in V$, by adding the commutator relations $[g, h] = 1$ for all $g \in G_v, h \in G_u$ such that $\{v, u\} \in E$. Graph products were first introduced by E. Green in her PhD thesis for arbitrary vertex groups \cite{green1990graph}. Further properties of graph products have been investigated by S. Hermiller and J. Meier  \cite{hermiller1995algorithms} and also by  D. Wise  and T. Hsu~\cite{hsu1999linear}. Recently they have become a topic of interest for their close connection to right angled buildings, CAT($0$) cube complexes and groups acting on trees (see for example \cite{Gibbins} and \cite{ruane2013cat}).
 
On the other hand, much progress has been made in analyzing finitely generated groups, particularly abelian groups, through their Cayley graphs, but less is known about the interplay between graph products of groups and the resulting Cayley graph structures. In \cite{loh2013finitely}, Clara L\"oh showed that two finitely generated abelian groups admit isomorphic Cayley graphs if and only if they have the same rank and their torsion parts have the same cardinality. This indicates that Cayley graphs of finitely generated abelian groups are rather rigid.
 
In this article, we focus on conditions under which graph products of  groups yield isomorphic Cayley graphs. To do this, we make sure that the generating sets are uniformly defined by their vertex groups. Firstly, we show that a bijection map between two sets of vertex groups implies a bijection between their graph products, then we explore Cayley graph isomorphisms. With this, we have the following main result: 

\begin{theorem*} \label{main} 
 {\it Let $\Gamma=(V,E)$ be a graph and for each vertex $v \in V$, let $G_v$ and $H_v$ be vertex groups, with symmetric generating sets 
$S_v\subset G_v-\{e_{G_{m}}\}$ and 
$T_v\subset H_v-\{e_{H_{m}}\}$. Suppose that $f_v: \Cay(G_v,S_v) \to \Cay(H_v,T_v)$ is a graph isomorphism for each $v\in V$.
Denote the graph products of $G_v$ and $H_v$ by $G_\Gamma$ and $H_\Gamma$ respectively, and put $S:=\bigcup_{v\in V}S_v$ and $T:=\bigcup_{v\in V}T_v $. Then there is a graph isomorphism
$$\Cay(G_\Gamma, S) \cong \Cay(H_\Gamma, T).$$ }
\end{theorem*}

This result addresses the isomorphism problem for Cayley graphs of graph products. Furthermore, we present a result about when the Cayley graphs of graph products are  isomorphic, but the groups are not:

\begin{cor}
{\em Let $\Gamma$ be a finite graph and let $G_\Gamma$ and $H_\Gamma$ be graph products of finite groups $G_v$ and $H_v$, with $|G_v|=|H_v|$. Then for suitable choices of vertex groups, $G_{\Gamma}\not\cong H_{\Gamma}$, while $\Cay(G_\Gamma,S_1)\cong \Cay(H_\Gamma,S_2)$.}
\end{cor}

With that, we obtain many examples for non-isomorphic finitely generated groups with isomorphic Cayley graphs.

\vspace{12pt}

\noindent\textbf{Acknowledgements:} The author sincerely thanks Prof. Linus Kramer for his invaluable guidance, advice, and constant support throughout the development of this article. His expertise and thoughtful feedback were crucial in shaping the direction of this work, and his encouragement greatly contributed to its completion.

\section{Graph products}

By a \emph{graph}, we just mean a pair $\Gamma =(V,E)$, where $V$ is a set, and we will call it the \emph{vertex set of $\Gamma$}, and $E$ is a set of $2$-element subsets of $V$, called the \emph{edge set of $\Gamma$}. Such graphs are also called simplicial graphs.
 Let $\Gamma = (V,E)$ be a simplicial graph. For each vertex $v$, let $G_v$ be a group, which we call the vertex group of $v$. We assume that vertex groups for different vertices have no elements in common. The graph product $G_{\Gamma}$ of a collection of vertex groups $G_{v}$, with respect to $\Gamma$, is the group obtained from the free product of $G_v$ by adding the relation $[g_v,g_u]=1$ for all $g_v \in G_v, g_u \in G_u$ such that $\{v,u\}$ is an edge of $\Gamma$. The group $G_\Gamma$ can be presented as follows:
$$G_{\Gamma} ={* G_v}/ \left< \left< [g_v,g_u]|g_v \in G_v, g_u \in G_u ~~ \text{and}~~\{u,v\} \in E \right> \right>$$

An important particular case of a graph product of groups is when all the vertex groups are cyclic. In the
case when they are infinite cyclic, one obtains the class of right-angled Artin groups, and when they have order 2,
one obtains the class of right-angled Coxeter groups.

 In order to develop a normal form for elements in a graph product as an analogue of a reduced word in a free product, we introduce a reduced word in a graph product. A word in a graph product is a sequence
$(g_1,\ldots,g_k)$, representing the element $g_1g_2\cdots g_k$ in the graph product $G_\Gamma$. 
The $g_k$ are elements of the vertex groups, which are called \emph{syllables}.
The following are transformations for a word $w=(g_1,\ldots,g_k)$ that do not change the value of the element represented by the word.
A word that cannot be shortened in this way is called a \emph{reduced word}. 
\begin{itemize}
	\item [(i)] Remove a syllable $g_i$ if $g_i =1$.
	\item [(ii)] Replace two consecutive syllables $g_i$ and $g_{i+1}$ 
    if they belong to the same vertex group $G_v$ with a single syllable $g_ig_{i+1}$ (i.e, the two syllables combine to become one syllable).
	\item[(iii)] (\emph{Syllable shuffling}) For consecutive syllables $g_i \in G_u, g_{i+1} \in G_v$ with $\{u,v\}$ an edge of $\Gamma$, interchange $g_i$ and $g_{i+1}$.
\end{itemize}

Note that $(iii)$ preserves the length of the word whereas $(i),(ii)$ decrease it by $1$. When we start with some word $w$ and apply finitely many transformations $(i)-(iii)$ we can obtain a \emph{reduced word} $w'$ representing the same element of the group $G_{\Gamma}$.

The following theorem about the normal form of graph products was first proved by E. Green  \cite[Theorem 3.9]{green1990graph} in her thesis.

\begin{thm}[The Normal Form Theorem for Graph Products] \label{normal}
Let $G_{\Gamma}$ be a graph product of the groups $G_v$ for $v \in V$. Each element $g \neq 1$ of $G_{\Gamma}$ can be expressed as a 
product $g=g_1g_2 \dots g_k$ where $w=(g_1,g_2, \dots, g_k)$
is a reduced word. This reduced word is unique up to syllable shuffling.
\end{thm}

\noindent The next Lemma is the first step
towards our main result.

\begin{lem}\label{lemma1}
 {\em Let $\Gamma=(V,E)$ be a graph and for each vertex $v \in V$,  let $G_v$ and $H_v$ be vertex groups.
Assume that for every vertex $v$, there is a
bijection $f_v:G_v\to H_v$, with $f_v(e_{G_v})=e_{H_v}$.
Denote the graph products of $G_v$ and $H_v$ by $G_{\Gamma}$ and $H_{\Gamma}$, respectively. Then we obtain a bijection $f : G_{\Gamma} \to H_{\Gamma}$ such that if $w=(g_1, \dots, g_k$) is a reduced word,  we have $f(w) = (f_{v_ {1}}(g_1),\dots, f_ {v_{m}}(g_k))$. Moreover, $f(w)$ is again a reduced word.}
\end{lem}

\begin{proof}
Let $w=(g_1, \dots, g_k$) be a reduced word. Then we
define the map $$f : G_{\Gamma} \to H_{\Gamma}$$  on reduced words by 
$$f(w) = (f_{v_ {1}}(g_1),\dots, f_ {v_{k}}(g_k))$$ 
 where $f_{v_{i}}: G_{v_{i}} \to H_{v_{i}}$ is the given bijection for each vertex $v_i$.
Since $f_v$ maps elements of $G_v$ to $H_v$, the map
$f$ preserves the structure of reduced words as we explain now. 
Consecutive  syllables $g_i$ and $g_{i+1}$ from different vertex groups remain in distinct vertex groups when mapped under $f$. The commutation relations between vertex groups in $G_\Gamma$ and $H_\Gamma$ are preserved, because if $G_{v_{i}}, G_{v_{i+1}}$ belong to adjacent vertices
$v_i,v_{i+1}$, they are mapped to vertex groups $H_{v_{i}}, H_{v_{i+1}}$ respectively. Therefore, shuffling $g_ig_{i+1}=g_{i+1}g_{i}$
in $G_\Gamma$ correspond to shuffling
$f_{v_{i}}(g_i)f_{v_{i+1}}(g_{i+1})=f_{v_{i+1}}(g_{i+1})f_{v_{i}}(g_{i})$ in $H_\Gamma$. Hence $f$ preserves syllable shuffling.
This shows that $f$ maps reduced words to reduced
words.
By Theorem \ref{normal}, the map
$f$ yields a well-defined map from
$G_\Gamma$ to $H_\Gamma$.

To show that $f$ is a bijection, we let $f_v'$
 denote the inverse of $f_v$ and then we define $f':H_\Gamma\to G_\Gamma$
 similarly as above, using the maps $f_v':H_v\to G_v$.
Let $(h_1,\ldots,h_k)$ be a reduced word where $h_i \in H_{v_{i}}$.

We now show that \( f \) and \( f' \) are inverses of each other. 
We have that

\begin{align*}
  f(f'(h_1,\ldots,h_k)) & = f\big(f'_{v_1}(h_1),\cdots, f'_{v_k}(h_k)\big)\\
  & = \big(f_{v_1}(f'_{v_1}(h_1)),\cdots f_{v_k}(f'_{v_k}(h_k))\big)\\
  &= (h_1,\ldots,h_k).  
\end{align*}
Thus $f\circ f'=\text{id}_{H_\Gamma}$.
Similiarly, $f'\circ f=\text{id}_{G_\Gamma}$.
Therefore, $f$ is bijective.
\end{proof}

\vspace{12pt}
\section{Cayley graphs}

Let $G$ be a group and let $S$ be a generating set of $G$. We assume that $S$ is symmetric, that is, $S=S^{-1}$
and it does not contain the identity.
Then the Cayley graph of $G$ with respect to the generating set $S$ is the graph $\Cay(G,S)$ whose set of vertices is $G$ and whose
set of edges is $\{\{g,gs\}|g \in G, s \in S\}$. Two vertices in a Cayley graph are thus adjacent if and only if they differ by right multiplication by an element of the generating set in question. This construction offers a visual way to explore the properties of the group:
the identity element $e_G$ serves as a central vertex, and the elements in $S$ are adjacent to it. 

\noindent We recall some properties of Cayley graphs from \cite[Remark 3.2.3]{loh2017geometric}.

\begin{rem}[Elementary properties of Cayley graphs]\label{elementary}\ 

\begin{itemize}
    \item [(i)] Cayley graphs are connected as every vertex $g$ can be reached from the vertex of the neutral element by walking along the edges corresponding to a presentation of minimal length of $g$ in terms of the given generators.
    \item[(ii)] Cayley graphs are regular in the sense that every vertex has the same number $|S|$ of neighbors.
    \item[(iii)] A Cayley graph is locally finite if and only if the generating set is finite; a graph is said to be locally finite if every vertex has only finitely many neighbors.
    \item[(iv)] The group $G$ acts on $\Cay(G,S)$ by left multiplication, and this
    action is transitive on the vertex set.
\end{itemize}
    
\end{rem}

\noindent Two Cayley graphs $\Cay(G,S)$ and $\Cay(H,T)$ 
are isomorphic if there is a \emph{graph isomorphism} between them, i.e, a bijection $f: \Cay(G,S) \to \Cay(H,T) $ such that for all $g_1,g_2 \in G$, $g_1$ and $g_2$ are adjacent in $\Cay(G,S)$ if and only if $f(g_1)$ and $f(g_2)$ are adjacent in $\Cay(H,T)$.

\begin{exa}
  \item [(i)] Let $G,H$ be two different finite groups having the same order, then $\Cay(G,S)\cong \Cay(H,T)$ where $S=G-\{e_G\}$ and $T=H-\{e_H\}$ are the respective generating sets.
  This is because the Cayley graphs are in this case complete graphs
  on the same number of vertices, and therefore isomorphic.
    \item [(ii)] In \cite{loh2013finitely}, L\"oh, as a consequence of her Theorem 1.3, showed that two finitely generated Abelian groups admit isomorphic Cayley graphs for finite generating sets if and only if they have the same rank and their torsion parts have the same cardinality.
\end{exa}

\noindent The following Lemma is the second step towards our main result.

\begin{lem} \label{lemma2}
{\em Let $G$ and $H$ be two groups with symmetric generating sets $S \subset G$ and $T \subset H$ not containing the identity.  Suppose that there exists a graph isomorphism $h:\Cay(G,S) \to \Cay(H,T)$ where $\Cay(G,S)$ and $\Cay(H,T)$ are the respective Cayley graphs. Then there exists a graph isomorphism $f:\Cay(G,S) \to \Cay(H,T)$ such that $f(e_G)=e_H$. Moreover $f(S)=T$.}
\end{lem}

\begin{proof}
    We define the map $f$ by $f(g)=ah(g)$ where $a=h(e_G)^{-1}$.
    Since $h:\Cay(G,S)\to \Cay(H,T)$ is a graph isomorphism and since left multiplication by
    $a$ is a graph automorphism of $\Cay(H,T)$ 
    by Remark \ref{elementary}(iv), the map $f$ is a graph
    isomorphism with $f(e_G)=e_H$.

    The vertices adjacent to $e_G$ are the elements of $S$, and the vertices
    adjacent to $e_H$ are the elements of $T$. Hence $f$ maps $S$ bijectively
    onto $T$.
\end{proof}

\vspace{12pt}

\noindent We now state and prove the final step of our main result.

\begin{thm} \label{main}
 Let $\Gamma=(V,E)$ be a graph and for each vertex $v \in V$, let $G_v$ and $H_v$ be
 vertex groups with symmetric
generating sets 
$S_v\subset G_v-\{e_{G_{m}}\}$ and 
$T_v\subset H_v-\{e_{H_{m}}\}$. Assume also that $f_v:\Cay(G_v,S_v) \to \Cay(H_v,T_v)$ is for each $v\in V$ a graph isomorphism.
Denote the graph products of $G_v$ and $H_v$ as  $G_\Gamma$ and $H_\Gamma$ respectively, and put $S=\bigcup_{v\in V}S_v$ and $T=\bigcup_{v\in V}T_v $. Then there is a graph isomorphism
$$\Cay(G_\Gamma, S) \cong \Cay(H_\Gamma, T).$$ 

\end{thm}

\begin{proof}
By Lemma \ref{lemma2} we may assume that each $f_v$
maps $e_{G_v}$ to $e_{H_v}$.
Let $w=(g_1, \dots ,g_k)$ be a reduced word in $G_\Gamma$. It follows from \cref{lemma1} that there exist a bijection $f:G_\Gamma \to H_\Gamma$ such that $f(g_1\cdots g_k)=f_{v_1}(g_1)\cdots f_{v_k}(g_k)$
for each reduced word $(g_1,\ldots,g_k)$.
This already gives us a bijection between the vertices of the
Cayley graphs, which maps $S$ bijectively onto $T$. 
 It remains to show that $f$  preserves adjacency.

First of all, we note that for each $s\in S$, we have
a reduced word $(s)$. 
In particular, $S$ is a symmetric generating set of $G_\Gamma$ not containing the identity, and similarly for $T$.

Let $g\in G_\Gamma$ and $s\in S$. we claim that
$f(g)$ and $f(gs)$ are adjacent in $\Cay(H_\Gamma,T)$ and we
distinguish three cases. For this, we consider all reduced words
$(g_1,\ldots,g_k)$ representing $g$. We also assume that $s\in S_v$,
for a vertex $v\in V$.
\begin{itemize}[leftmargin=15pt, labelwidth=2em, labelsep=0.5em, align=left]
\renewcommand\labelitemi{}
\item [Case 1:]No reduced word $(g_1,\ldots,g_k)$
                representing $g$ ends with a last syllable $g_k\in G_v$. We claim that $(g_1,\ldots,g_k,s)$ is a reduced word representing $gs$. 
                Suppose that  $(g_1,\ldots,g_k,s)$ is not a reduced word, then 
                $s$ undergoes shuffling to the left as
                $(g_1,\ldots,g_j,s, \ldots g_k)$ 
                so that it combines with a syllable $g_j\in G_v$
                for some $j<k$. But then we could also shuffle $g_j$ to the 
                right-most position, contradicting our assumption.
                Hence $(g_1,\ldots,g_k,s)$ is a reduced word. Then $f(g)=f_{v_1}(g_1)\cdots f_{v_k}(g_k)$
                and $f(gs)=f_{v_1}(g_1)\cdots f_{v_k}(g_k)f_v(s)=f(g)f(s)$,
                with $f(s)\in T$. This shows that $f(g)$ and $f(gs)$ are adjacent.
            
\item [Case 2:] The element $g$ can be written as a reduced word 
               $(g_1,\ldots,g_k)$ with last syllable $g_k=s^{-1}$. Then $(g_1,\ldots,g_{k-1})$
               is a reduced word representing $gs$ and 
               $f(gs)=f_{v_1}(g_1)\cdots f_{{v}_{k-1}}(g_{k-1})=f(g)f(s)$.
                Hence $f(gs)=f(g)f(s)$ also in this case, and $f(g)$ is adjacent to $f(gs)$.
\item [Case 3:] The element $g$ can be written as a reduced word 
              $(g_1,\ldots,g_k)$ with the last syllable $g_k\in G_v$,  with $g_k\neq s^{-1}$. Then $(g_1,\ldots,g_ks)$ is a reduced word representing $gs$, since the last syllable $g_k$ combines with $s$ in $(g_1,\ldots,g_ks)$ to be one syllable $g_ks$ and
              $g_ks$ cannot be
              shuffled next to any $g_j$, with $j<k$ and $g_j\in G_v$. This is because the consecutive syllables do not belong to adjacent vertex groups and having $g_ks$ as the last syllable does not change that. 
             From this it follows that
            $f(g_1\cdots g_ks)=f_v(g_1) \cdots f_{{v}_{k}}(g_ks)$. Then
             \begin{align*}
                f(g_1 \cdots g_k) &=f(g_1 \cdots g_{k-1})f_{v}(g_{k}), \\ 
                f(g_1 \cdots g_ks) &= f(g_1 \cdots g_{k-1})f_{v}(g_{k}s).
            \end{align*}
            Since the map $f_v$ is an isomorphism of Cayley graphs according to Lemma \ref{lemma2}, $f_{v}(g_k)$ being
            adjacent to $f_{v}(g_ks)$ implies that $f_{v}(g_ks) = f_{v}(g_k)t$ for some $t \in T_v$. Hence 
            $$f(g_1 \cdots g_{k-1})f_{v}(g_{k}s) = f(g_1 \cdots g_{k-1})f_{v}(g_{k})t.$$ 
            It follows that $f(gs)= f(g)t$, as required.
\end{itemize}
These are all possible cases. Hence $f$ maps
vertices adjacent to $g$ to vertices adjacent to $f(g)$.
The same holds for the inverse map $f':H_\Gamma\to G_\Gamma$.
Thus $f$ is an isomorphism of graphs.
\end{proof}

\vspace{12pt}

The main case of interest of Theorem~\ref{main} is when the graph $\Gamma$ is finite and when the generating sets $S_v$ are also finite.
Then the Cayley graph of $G_\Gamma$ has finite degree at each vertex.
\begin{exa} \label{exa2}
Let $\Gamma$ be a finite graph and for each vertex $v$, let $G_v$ and $H_v$ be
finite groups with $|G_v|=|H_v|$. Then the graph products $G_\Gamma$ and $H_\Gamma$
have isomorphic Cayley graphs with respect to the finite generating sets
$S_1=\bigcup_{v\in V}(G_v-\{e_{G_{m}}\})$ and 
$S_2=\bigcup_{v\in V}(H_v-\{e_{H_{m}}\})$, i.e,
$\Cay(G_\Gamma,S_1)\cong \Cay(H_\Gamma,S_2).$

\noindent In particular, $G_\Gamma$ and $H_\Gamma$ are isometric with respect to their word metrics.
However, for suitable choices of the vertex groups $G_v,H_v$, the groups $G_\Gamma$ and 
$H_\Gamma$ are not isomorphic.

\vspace{12pt}

\noindent The condition $|G_v|=|H_v|$ cannot be omitted, as the following example shows. Let $\Gamma$ denote the graph with two vertices $u,v$ and no edge. Put $G_u=G_v=\mathbb Z/2$ and $H_u=H_v=\mathbb Z/3$.
Then the Cayley graphs of the graph products $G_\Gamma=\mathbb Z/2*\mathbb Z/2$ and $H_\Gamma=\mathbb Z/3*\mathbb Z/3$ are not isomorphic and not even quasi-isometric, because they have 
different numbers of ends, $|ends(G_\Gamma)|=2$ and $|ends(H_\Gamma)|=\infty$.
The first Cayley graph is an infinite line, while the second Cayley graph is a trivalent tree.

\end{exa}

\noindent The following is a result that was proved by Genevois and Martin's 
\cite[Lemma 3.18]{genevois2019automorphisms} about the maximal finite subgroup of graph products of finite groups. It will help us to show that certain graph products of finite groups are not isomorphic. A clique in a graph is a subgraph which is complete.
\begin{lem}\label{lemmax}
{\em Let $G_{\Gamma}$ be a graph product of finite groups. The maximal finite subgroups of $G_\Gamma$ are exactly the $g (G_\Lambda) g^{-1}$ where $g\in G_{\Gamma}$ and for all maximal cliques $\Lambda\subseteq\Gamma$, $G_\Lambda$ is a subgroup generated by vertex groups corresponding to the vertices of $\Lambda$. Also, $G_\Lambda$ is isomorphic to the direct product of the vertex groups $G_v$ whose vertices $v$ are in $\Lambda$.}
 \end{lem}

\noindent We now establish the non-isomorphism condition of graph products by analyzing their maximal finite subgroups. This will help us see that graph products can have isomorphic Cayley graphs, despite being non-isomorphic themselves.

 \begin{thm}
 Let $\Gamma$ be a finite graph and let $G_\Gamma$ and $H_\Gamma$ be graph products of finite groups.
 Assume that for all maximal cliques $\Lambda,\Delta \in \Gamma$, $G_\Lambda \not\cong H_\Delta$
 then, $G_{\Gamma}\not\cong H_{\Gamma}$. 
\end{thm}
 
\begin{proof}
    By Lemma \ref{lemmax}, the maximal finite subgroups of $G_\Gamma$ and $H_\Gamma$ are not isomorphic,
   hence $G_\Gamma$ and $H_\Gamma$ are not isomorphic. 
\end{proof}

\begin{cor}
{\em Let $\Gamma$ be a finite graph and let $G_\Gamma$ and $H_\Gamma$ be graph products of finite groups $G_v$ and $H_v$, with $|G_v|=|H_v|$. Then for suitable choices of vertex groups, $G_{\Gamma}\not\cong H_{\Gamma}$, while $\Cay(G_\Gamma,S_1)\cong \Cay(H_\Gamma,S_2)$.}
\end{cor}

\noindent This applies in particular to example \ref{exacay} below.

\noindent The graph product is a free product if the finite graph $\Gamma=(V,E)$ has no edges between the vertices. Let $V$ be a finite set and $\{G_v| v \in V\}$ be a finite set of groups and $G= *G_v$ denote the free product of groups $G_v$, then $G_v$ are called the free factors of $G$.
 The following is an example of free products of finite groups with isomorphic Cayley Graphs.
\begin{exa} \label{exacay}
    Let $G_1 = \mathbb{Z}_4$   and $G_2 =S_3$ be free factors then $G= G_1 * G_2 = \mathbb{Z}_4 * S_3$, where $\mathbb{Z}_4 = \langle a^4 = e\rangle$  and we denote the elements of $S_3$ by $\{e,b_1,b_2,b_3,b_4,b_5\}$. Let $S_1 = ((\mathbb{Z}_4 - \{e_{\mathbb{Z}_4}\}) \cup (S_3 - \{e_{S_3}\}))$ be the generating set of $G$. Then the Cayley graph of $G$ is denoted by $\Cay(G,S_1)$.
    
    Also, Let  $H_1 =\mathbb{Z}_2 \times \mathbb{Z}_2$ and $H_2 =\mathbb{Z}_6=\langle d^6=e\rangle $ be the free factors then $H = H_1 * H_2 =(\mathbb{Z}_2 \times \mathbb{Z}_2) * \mathbb{Z}_6$. Let $S_2 =(((\mathbb{Z}_2 \times \mathbb{Z}_2) - \{e_{(\mathbb{Z}_2 \times \mathbb{Z}_2})\}) \cup (\mathbb{Z}_6 - \{e_{\mathbb{Z}_6}\}))$ be the generating set of $H$. We denote the Cayley graph of $H$ by $\Cay(H,S_2)$.

Since $|\mathbb{Z}_4 |=|\mathbb{Z}_2 \times \mathbb{Z}_2|$ and $|S_3|=|\mathbb{Z}_6|$, it follows from Theorem \ref{main} that $\Cay(G,S_1) \cong \Cay(H,S_2)$.

Figure \ref{fig:Cayley Graph} shows part of the Cayley graph of $G$ which is isomorphic to the Cayley graph of $H$.

  \begin{center}
\begin{figure}[h!]{}
 \includegraphics[scale=0.51]{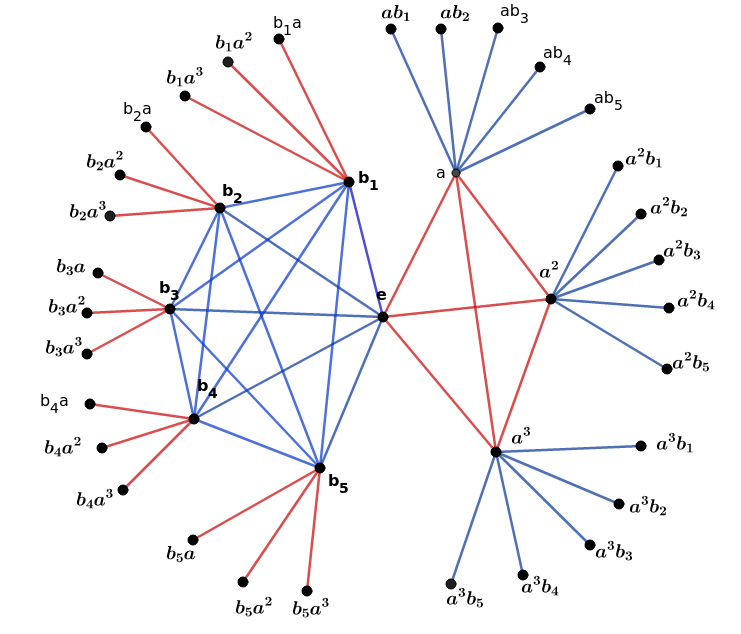}
   \caption{Cayley graph of $\mathbb{Z}_4 * S_3$ (The blue edges represent 
    $S_3$ while the red ones represent $\mathbb{Z}_4$.)}
 \label{fig:Cayley Graph}
\end{figure}
\end{center}
\end{exa}

\bibliographystyle{plain}

\end{document}